\documentclass[letterpaper,11pt]{amsart}
 
\usepackage{amsfonts,amsthm,latexsym,amsmath,amssymb,amscd,amsmath, mathrsfs, geometry, enumerate}
\usepackage{color}
\usepackage[utf8]{inputenc}
\usepackage{picture}
\usepackage{tikz}
\usepackage{graphicx}
\usepackage{verbatim}

 \newtheorem{thm}{Theorem}[section]
 \newtheorem{cor}[thm]{Corollary}
 \newtheorem{lem}[thm]{Lemma}
 \newtheorem{prop}[thm]{Proposition}
 \theoremstyle{definition}
 \newtheorem{defn}[thm]{Definition}
 \theoremstyle{remark}
 \newtheorem{rem}[thm]{Remark}
 \newtheorem*{ex}{Example}
 \numberwithin{equation}{section}
 
 \newtheorem{cl}[thm]{Claim}

 \newcommand{\M}{\mathcal{M}}

 \newcommand{\X}{\mathbb{X}}

 \newcommand{\XX}{\mathbb{X}}
 \newcommand{\PP}{\mathbb{P}}

\title[Waring-like decompositions of polynomials, 1]
 {Waring-like decompositions of polynomials, 1}

\author[M.V.  Catalisano]{Maria V. Catalisano}
\address[M.V. Catalisano]{Dipartimento di Ingegneria Meccanica, Energetica, Gestionale e dei Trasporti, Universit\`{a} degli studi di Genova, Genoa, Italy.}
\email{catalisano@diptem.unige.it}

\author[L. Chiantini]{Luca Chiantini}
\address[L. Chiantini]{Dipartimento di Ingegneria dell'Informazione e Scienze Matematiche,  Universit\`{a} di Siena, Italy.}
\email{luca.chiantini@unisi.it}

\author[A.V. Geramita]{Anthony V. Geramita$^*$}
\address[A.V. Geramita]{Department of Mathematics and Statistics, Queen's University, King\-ston, Ontario, Canada}
\email{Anthony.Geramita@gmail.com \\ geramita@dima.unige.it  }
\thanks{\it $^*$During the submission of this paper, Tony Geramita passed away. In memory of a friend, besides a colleague and a mentor, the three remaining authors dedicate this paper to him.}

\author[A.Oneto]{Alessandro Oneto}
\address[A. Oneto]{Inria Sophia Antipolis M\'editerran\'ee, Sophia Antipolis, France.}
\email{alessandro.oneto@inria.fr}

\subjclass[2010]{14Q20, 13P05, 14M99, 14Q15}

\keywords{Waring problems, polynomials, secant varieties}

 \dedicatory{In memory of Tony Geramita (1942--2016)}
\begin{document}

\begin{abstract}
 Let $f$ be a homogeneous form of degree $d$ in $n$ variables. A Waring decomposition of $f$ is a way to express $f$ as a sum of $d^{th}$ powers of linear forms. In this paper we consider the decompositions of a form as a sum of expressions, each of which is a fixed monomial evaluated at  linear forms. 
\end{abstract}

%%% ----------------------------------------------------------------------
\maketitle
%%% ----------------------------------------------------------------------
%\tableofcontents

\section{Introduction}
Let $f\in k[x_1, \dots, x_n] = R = \oplus_{i \geq 0 } R_i$, ($k = \bar k$  and $n \geq 2$). The necessities of a given problem involving $f$ often make it useful to have different ways to {\it decompose} $f$.  E.g. in many computations it is useful to express a polynomial as a sum of monomials ordered in a specific way.  Different applications call for different kinds of decompositions.  The following papers give some interesting uses of some non-standard decompositions (see, e.g. \cite{henrion}, \cite{comon}, \cite{ComMour}, \cite{ChGe}).

One particular kind of decomposition that has received a great deal of attention is the {\it Waring decomposition}. This decomposition asks us to write
$f \in R_d$  in an efficient way as 
$$f = \sum _{i=1} ^{s} L_i^d
$$
where each $L_i$ is a linear form.  There is an extensive literature on this decomposition with many interesting applications (see \cite{ComMour},\cite{AH95},\cite{IaKa},\cite{LandsbergTeitler2010},
\cite{carcatgermonomi}).

In this paper we want to propose an extension of the notion of Waring decomposition. To explain the idea we introduce the following notation.

Fix an integer $d$. Let $ {\mathcal P}_r= \mathbb Z [Z_1, \dots, Z_r]$ and let $\mathcal M_{r,d}$ be the subset of all monomials such that
$$M = Z_1^{d_1} \cdots Z_r^{d_r}  $$
and  
\begin {itemize}
\item [$i)$] $d_i >0$, for all $ i$,
\item[$ii)$] $d_1+ \cdots +d_r =d$.
\end {itemize}

Of course, $i)$ and $ii)$  imply $r \leq d$.

\begin {defn} Let $M$ be as above and let $f \in R_d$. 

 $ i) $ An {\it $M$-decomposition of $f$ having length $s$} is an expression of the form 
$$f= \sum _{j=1}^s L_{1,j}^{d_1}\cdot L_{2,j}^{d_2} \cdots L_{r,j}^{d_r},$$
where the $L_{i,j}$ are linear forms.

$ii)$ The {\it $M$-rank} of $f$ is the least integer $s$ such that $f$ has an $M$-decomposition of length $s$.
\end {defn}

%%%%%%%%%%%%%%%%%%%%

\begin{rem} \label {casinoti}
$i)$ If $M = Z_1^d$ then the $M$-rank of $f$ is known as the {\it Waring rank} of $f$.

$ii)$ Every $f \in R_d$ has an $M = Z_1^{d_1}\cdots Z_r^{d_r}$-decomposition of finite length for any choice of $M$. This is immediate from the fact that every $f \in R_d$ has finite Waring rank and
$$L^d = L^{d_1} \cdots  L^{d_r}.$$

\medskip

$iii)$ 
Two other special cases have received a great deal of attention recently: when $M=  Z_1 \cdots Z_r $  ($r \geq 2$) then the forms  in $R_r$ which have $M$-rank equal to 1  are called {\it split forms} or {\it completely decomposable forms}.  $M$-decompositions for these $M$ are considered in \cite{arrondobernardi},  \cite{abo}, \cite{magnifici}, \cite {shin2012} and \cite{torrance}.

When $M=  Z_1 ^{d-1} Z_2$, $M$-decompositions were first considered in \cite{CGG0} and then in \cite{BCGI2009} and \cite{ballico2005}.  They arose naturally from a consideration of the secant varieties of the tangential varieties to the Veronese varieties.  The recent work \cite{aboV} is a major contribution to this decomposition problem.
\medskip

$iv)$ For the case of binary forms (i.e. $n=2$) this problem has its roots in the very foundations of modern algebra.  Since the $M$-rank of a binary form is invariant under the usual $SL_2(k)$ action on $k[x_1,x_2]$, we see in the works of Cayley, Salmon, Sylvester (\cite {Cayley}, \cite {Sal3}, \cite{Syl}) the search for the invariants which characterize forms of $M$-rank 1 for special choices of $M$.

A modern treatment of these classical investigations (as well as advances on them) can be found in the lovely papers of Chipalkatti (see \cite {Chi02}, \cite {Chi03}, \cite {Chi04}, \cite {Chi04bis}, \cite {CaCh}).
The bibliographies in these papers give a quick entry into the classical literature on the subject.
\medskip

$v)$ The $M$-rank of a form obviously depends on $M$.  E.g. recall that the Waring rank of $xyz$ in $k[x,y,z]$ is 4 (see e.g. \cite{carcatgermonomi}) while if $M= Z_1^2Z_2$, the $M$-rank of $xyz$ is 2. To see this note that the $M$-rank is bigger than 1, but 
$$ 4xyz = x((y+z)^2-(y-z)^2)
.$$
\end{rem}

\medskip
\medskip
 
There is a geometric way of considering the problem of finding the $M$-rank of a polynomial $f \in R_d$, $M \in \mathcal M_{r,d}$, $M = Z_1^{d_1} \cdots Z_r ^{d_r}$.  

Let $\PP[R_1]$ be the projective space based on the $k$-vector space $R_1$. We define the morphism
$$
\varphi _M :\ \underbrace { \PP[R_1] \times \cdots \times \PP[R_1] } _{r \hbox { times}}\longrightarrow \PP[R_d]
\simeq \PP ^ { {d+n-1 \choose n-1 } -1} 
$$
by
$$\varphi _M ([L_1], \ldots, [L_r]) = [L_1^{d_1} L_2 ^{d_2} \cdots L_r^{d_r}],
$$
and denote the  image of $\varphi _M$ by $\mathbb X_M$.

\begin{rem} $i)$ Notice that when we have $ M= Z_1^d$, then $\mathbb X_M$ is precisely the $d^{th}$ Veronese embedding of $\PP(R_1) = \PP^{n-1}$ in $\PP(R_d)$.

$ii)$ In general, the variety $\mathbb X_M$ is a projection of an appropriate Segre-Veronese variety. However, we will not use that fact in this paper.

$iii)$ The equations which define the Veronese variety are well-known (see e.g. \cite{Pucci}). It would be interesting to find equations for the variety $\mathbb X_M$ when $M \neq Z_1^d$. (See, however, \cite{briand} for the variety of split forms).

\end{rem}

If $\mathbb X \subseteq \mathbb P^t$ is any projective variety, then 
$$ \sigma_s (\mathbb X) := \overline {\{ P \in \mathbb P^t\  | \ P \in < P_1, \ldots, P_s>, P_i \in \mathbb X ,\}}$$
is the {\it $s^{th}$-secant variety}  of $\mathbb X$.
$$ \sigma_1 (\mathbb X) = \mathbb X,$$
$$ \sigma_2 (\mathbb X) ={\hbox { secant line variety to }}  \mathbb X, \ \hbox { etc}.$$

\begin{rem} When ${\mathbb X} = {\mathbb X}_M$ and $M = Z_1^d$, so $\mathbb X$ is the $d^{th}$ Veronese embedding of $\PP^n$   then, if $F \in R_d$ has Waring rank $s$ we  have  $[F] \in \sigma_{s}({\mathbb X})$.  In particular, if $s$ is the least integer for which $\sigma_s({\mathbb X})= \PP(R_d)$ then the generic element $[F] \in \PP(R_d)$ has Waring rank $s$.

This is the fundamental connection between the algebraic problem of finding the Waring rank of a generic form and the geometric problem of finding the dimensions of secant varieties to Veronese varieties.
\end{rem}

\medskip
One can ask about the dimensions of the secant varieties of any projective variety.    More precisely: given $ \mathbb X \subset \PP^t$ what is $\dim  \sigma_s(\mathbb X)$ for $s \geq 2$?

There is a reasonable guess which gives an upper bound for $\dim  \sigma_s(\mathbb X)$.  It is obtained by counting parameters and observing that $\sigma_s(\mathbb X) \subseteq \mathbb P^t$, namely
\begin{equation}\label{eq:inequality expected dimension}
\dim  \sigma_s(\mathbb X) \leq \min \{ s \dim \mathbb X + s -1, t\}.  
\end{equation}
If we have equality in \eqref{eq:inequality expected dimension} for some $s$, then we say that $ \sigma_s(\mathbb X)$ has the {\it expected dimension}, while if  \eqref{eq:inequality expected dimension}  is a strict inequality for some $s$, we say that $ \sigma_s(\mathbb X)$ is  {\it defective} and the difference 
$$ \min \{ s \dim \mathbb X + s -1, t\} - \dim  \sigma_s(\mathbb X),$$
is called the $s$-{\it defect} of $\mathbb X$.

If $\mathbb X \subseteq \mathbb P^t$ is non-degenerate then $\sigma_s(\mathbb X) =\mathbb P^t$ for some $s$ and so the $s$-{defect} is eventually zero, for all $s \gg 0$. 

The particular problem we consider in this paper is that of finding  $\dim  \sigma_s(\mathbb X_M)$ for $R= k[x_1, \ldots.x_n]$ and any $M \in \mathcal M_{r,d}$.

The paper is organized in the following way. In Section 2 we recall Terracini's Lemma, which is our main tool in fi
nding $\dim  \sigma_s(\mathbb X_M)$.
Terracini's Lemma needs a description of the tangent space at a general point of $\mathbb X_M$. This tangent space corresponds to a vector space which is the graded piece of an ideal $I$ in $R$.  We use information about this ideal to find the dimensions we need. 

 In the third section we find the dimensions of all the secant varieties of $\XX_M$ for any $M \in {\mathcal M}_{r,d}$, for any $r$ and any $d$, in case $n=2$, i.e. for binary forms.  We find that there are no defective secant varieties in this case.

In the fourth section we find the dimensions of the secant line varieties of $\X_M$ for any $n$ and for any $M \in {\mathcal M}_{r,d}$ for any $r$ and any $d$.  In this family of cases we find exactly one defective secant line variety.

In the final section  we use  results of the previous sections and specialization in order to prove the non-defectiveness of new secant varieties.

\section {Preliminaries}
Since the cases $r=1$ and $d=2$  were already treated and solved (see Remark \ref {casinoti}), in this paper  we  assume $r\geq 2$ and $d \geq 3$.

We begin by recalling the Lemma of Terracini \cite {Terracini}.
\begin{lem} Let $\mathbb X \subseteq \mathbb P^t$ be a projective variety and let
$P \in \sigma_s(\mathbb X)$ be a general point.  If 
$$P \in <P_1, \ldots, P_s>$$
where $P_1, \ldots, P_s$ are general points of  $\mathbb X$, then the tangent space to $\sigma_s(\mathbb X)$
at $P$ is the linear span of tangent spaces to $\mathbb X$ at $P_1, \ldots, P_s$, i.e.,
$$T_P(\sigma_s(\mathbb X) )= <T_{P_1}(\mathbb X), \ldots, T_{P_s}(\mathbb X)>.$$
\end{lem}
To apply Terracini's Lemma to our situation we first need to calculate $T_{P}(\mathbb X_M)$ for a general $P$ in $ \mathbb X_M$.

\begin{prop} \label {spaziotang}
Let $R= k[x_1,\ldots,x_n]$ and let $M \in \mathcal M_{r,d}$, 
$$M= Z_1^{d_1}\cdots Z_r^{d_r}.$$
Let $L_1$,...,$L_r$ be general linear forms in $R_1$ so that $P= [L_1^{d_1}\cdots L_r^{d_r}]$ is a general point of $\mathbb X_M= \varphi_M(
\underbrace { \mathbb P^{n-1} \times \cdots \times \mathbb P^{n-1}} _{r \hbox { times}}) \subseteq \mathbb P^{{d+n-1 \choose n-1}-1}$.

If
$$F= L_1^{d_1}\cdots L_r^{d_r} \hbox { and }  I_P =\left ({F / L_1}, \ldots,{F / L_r}\right ) = \oplus_{j\geq 0}(I_P)_j$$
then
$$T_P( \mathbb X_M) = \mathbb P(  \left ( I_P \right ) _d ).$$
\end{prop}

\begin {proof}
Since $P = \varphi_M ([L_1],\ldots ,[L_r])$, the image of a line through the point 
$([L_1],\ldots ,[L_r])$ in the direction $([\widetilde L_1], \ldots , [\widetilde L_r])$ is the curve on the  variety $\mathbb X_M$ whose points are parameterized by 
$$
[ (L_1 + \lambda \widetilde L_1)^{d_1} \cdots (L_r +  \lambda \widetilde L_r)^{d_r}].$$
The tangent vector to this curve at $P$ is given by the coefficient of $\lambda$ in this last expression, that is,
$$ \left({F / L_1} \right) \widetilde L_1 + \cdots +   \left({F / L_r} \right) \widetilde L_r    .$$
These, for varying choices of the $\widetilde L_i$, give that the tangent vectors at $P$ are precisely the degree $d$ piece of the ideal generated by the $F / L_i$.

\end{proof}

\begin {rem} \label{remark1} 
$i)$
Note that if $M= Z_1^{d_1}\cdots Z_r^{d_r}$ and we are considering forms in $k[x_1,\dots,x_n]$, then
$\mathbb X_M= \varphi_M(
\underbrace { \mathbb P^{n-1} \times \cdots \times \mathbb P^{n-1}} _{r \hbox { times}})$,  Since the generic fibre of $\varphi _M$ is finite, 
the dimension of  $T_P( \mathbb X_M) $ is $r(n-1)$.

$ii)$
It is easy to see that if $F$ and the $L_i$ are as above, then 
 $$I_P =\left (F / L_1, \ldots, F / L_r\right ) $$
 \begin{equation}\label{eq:ideal}
 = (L_1^{d_1-1} \cdots L_r^{d_r-1}) \cdot 
( L_2L_3 \cdots L_r, L_1L_3 \cdots L_r,\ldots,  L_1L_2 \cdots  L_{r-1}) \ .
\end{equation}

\end{rem}

\begin {cor}[for binary forms]  \label{potenza2var} Let $R = k[x_1,x_2]$, $M \in {\mathcal M}_{r,d}$ 
{\rm(}$r \geq 2${\rm)}
$$ M = Z_1^{d_1}\cdots Z_r^{d_r}.$$
If  $L_1,L_2,  \ldots , L_r$ are general linear forms in $R_1$, $P=[L_1^{d_1}\ldots L_r^{d_r}] \in {\mathbb X}_M$ and we set 
$I^\prime$ to be the principal ideal
$$I^\prime:= (L_1^{d_1-1}\cdots L_r^{d_r-1})$$
then we have
$$T_P({\mathbb X}_M)= \PP(I^\prime_d).$$  
 
\end {cor}

\begin {proof} In view of equation \eqref{eq:ideal} above we first consider the ideal 
$$
J:=(L_2L_3\cdots L_r, L_1L_3\cdots L_r, \ldots , L_1L_2\cdots L_{r-1} ).
$$

\begin{cl}\label{claim} $J = (x_1,x_2)^{r-1}$
\end{cl}

\begin{proof} (of the Claim)  
By induction on $r$. Obvious for $r=2$ so let let $r>2$. Since
$$J = (L_r\cdot (L_2L_3 \cdots L_{r-1} , L_1L_3 \cdots L_{r-1} ,\ldots,  L_1L_2 \cdots  L_{r-2}) , 
L_1L_2 \cdots  L_{r-1}),$$
 we have, by the induction hypothesis, that 
$$J = (L_r\cdot (x_1,x_2)^{r-2} ,  L_1L_2 \cdots  L_{r-1}).$$
Since $R = k[x_1,x_2]$, $L_1L_2 \cdots  L_{r-1}$ is a general form in $R_{r-1}$, hence not in the space $L_r(x_1,x_2)^{r-2}$.  This last implies that
 $\dim J_{r-1} = r$.  Since the ideal $J$ begins in degree $r-1$ we are done with the proof of the claim. 
\end{proof}

Now, using Claim \ref{claim}, equation \eqref{eq:ideal} and Proposition \ref{spaziotang}
we have that 
$$
I_P = (L_1^{d_1-1}\cdots L_r^{d_r-1})(x_1,x_2)^{r-1} \subseteq (L_1^{d_1-1}\cdots L_r^{d_r-1})
$$
and since $(x_1,x_2)^{r-1} = \oplus_{j \geq r-1}R_j$, we have that 
$$
(I_P)_d = (L_1^{d_1-1}\cdots L_r^{d_r-1})_d 
$$

\end{proof}

\begin {cor}\label {tangrvar}
Let $M= Z_1^{d_1}\cdots Z_r^{d_r} \in \mathcal M_{r,d}$.
Let  
$P = \varphi_M([L_1] , \cdots , [L_r])$ be a general point of $\mathbb X_M$, where the $L_i$ are general 
linear forms in $k[x_1,\ldots,x_n]$ and let $I_P$ be as in Proposition \ref {spaziotang}. 
Then,
$$I_P = (L_1^{d_1-1} \cdots L_r^{d_r-1}) \cap (\cap_{1 \leq i <j\leq r} (L_i,L_j)^{d_i+d_j-1} )
.$$

\end {cor}

\begin {proof}By Remark \ref{remark1} we need to prove that the two ideals
$$
I_P=(L_1^{d_1-1}\cdots L_r^{d_r-1})\cdot (L_2L_3\cdots L_r, \ldots , L_1\cdots L_{r-1})
$$
and
$$
J=(L_1^{d_1-1} \cdots L_r^{d_r-1}) \cap (\cap_{1 \leq i <j\leq r} (L_i,L_j)^{d_i+d_j-1} )
$$
are equal.

Since each generator of $I_P$ is in $J$, we have $I_P \subseteq J$.

Now let $h= L_1^{d_1-1}\cdots L_r^{d_r-1}$, and suppose that $f=hg \in J$.  Since the $L_i$ are general linear forms we have 
$$
h \in (L_i,L_j)^{d_i + d_j -2} \hbox{ \ \ for every $1 \leq i<j \leq r$},
$$
but
$$
h \notin (L_i,L_j)^{d_i + d_j - 1} \hbox{ \ \ for every $1 \leq i<j \leq r$},
$$
so $g \in rad((L_i,L_j)^{d_i + d_j -1} )= (L_i,L_j)$ for all $1 \leq i<j \leq r$.

Since $\bigcap_{1 \leq i<j \leq r} (L_i, L_j) = (L_2\cdots L_r, \ldots , L_1\cdots L_{r-1})$ we are done.

\end{proof}

\section{The Binary Case}

For binary forms there is a simple theorem covering all cases.

\begin{thm}\label{Binary}
Let $R = k[x,y] = \oplus_{j\geq 0} R_j$ and let $M = Z_1^{d_1}\cdots Z_r^{d_r} \in {\mathcal M}_{r,d}$ for any $r$ and any $d$ with $r \leq d$.
Then $\sigma_s({\mathbb X}_M)$ has the expected dimension for every $s$, i.e.
$$
\dim \sigma_s({\mathbb X}_M) = \min\{s\dim {\mathbb X}_M + (s-1), d \} = \min \{sr +s-1, d \}
$$
for every $s$ and every $M$.
\end{thm}

\begin{proof}  Since every form in $R$ of degree $d$ splits as a product of linear forms and the general form of degree $d$ has no repeated factors, we conlude that for $M = Z_1\cdots Z_d$, ${\mathbb X}_M = \PP(R_d) = \PP^d$.  This takes care of the case $r = d$.  Now, for the rest of the proof, assume that $r < d$.

By Corollary \ref{potenza2var} we know that if $P=[L_1^{d_1}\cdots L_r^{d_r}]$ is a general point of ${\mathbb X}_M$ (where $L_1, \ldots , L_r$ are general in $R_1$) then 
$$
T_P({\mathbb X}_M) = \PP( (I^\prime_P)_d) \ \ \hbox{ where }\ \ I^{\prime}_P = (L_1^{d_1-1}\cdots L_r^{d_r-1}) .
$$

So, by Terracini's Lemma, if $P_1, \ldots , P_s$ are a set of $s$ general points of ${\mathbb X}_M$ then
$$
\dim(\sigma_s({\mathbb X}_M)) = \dim_k (I^\prime_{P_1} + \cdots  + I^\prime_{P_s} ) _d - 1
$$
where if $P_j = [L_{j1}^{d_1}\cdots L_{jr}^{d_r}]$ then $I^\prime_{P_j} = (L_{j1}^{d_1-1}\cdots L_{jr}^{d_r - 1} )$.

However,
 \cite[Cor.2.3]{GeSc} states that for the special points $Q_1 =[H_1^d], \ldots , Q_s =[H_s^d] \in {\mathbb X}_M$ (where the $H_j$ are general in $R_1$) and for the ideal $J = (H_1^{d-r}, \ldots , H_s^{d - r})$ we have
$$
\dim_k  (J_d) = \min \{ d+1, s(r+1) \} = \min\{ d,sr+(s-1) \} + 1 .
$$

Since we know that $\dim \sigma_s({\mathbb X}_M) \leq \min \{ d, sr + (s-1) \}$, it  follows (by semicontinuity) that 
$$
    \dim (I^\prime_{P_1} + \cdots  + I^\prime_{P_s} ) _d = \min\{ d,sr+(s-1) \} + 1 .
$$
and so $\sigma_s({\mathbb X}_M)$ always has the expected dimension.
\end{proof}

\section{The Secant Line Varieties to $\mathbb X_M$ }

In this section we will find  the dimensions of the secant line varieties of $\mathbb X_M$ for every $M \in {\mathcal M}_{r,d}$ and for every polynomial ring $R = k[x_1, \ldots , x_n]$.  

When $r=1$ this is one part of the complete solution to Waring's Problem solved by Alexander and Hirschowitz in \cite{AH95}, so we will assume that $r \geq 2$.  In the previous section we solved this problem for $n = 2$, so we may now  assume that $n \geq 3$.  Recall that we are also assuming $r\geq 2$ and $d \geq 3$.

The main result of this paper is the following theorem.

\begin {thm} \label{sec2varn}
Let $R= k[x_1,\ldots,x_n]$, let $M \in \mathcal M_{r,d}$,  
$n \geq 3$ , $r\geq 2$, $d \geq 3$,

$$M= Z_1^{d_1}\cdots Z_r^{d_r}.$$
Then $\sigma_2 (\mathbb X_M)$  is not defective, except for $M= Z_1^2 Z_2$ and $n=3$. For this last case $\mathbb X_M$ has 2-defect equal to 1.
\end {thm}

\begin {proof} We always have $d \geq r$. The case $d=r$ is covered in  \cite[Theorem 4.4] {shin2012} so we may as well assume that $d>r$ also.

By Terracini's Lemma we need to find the vector space dimension of 
$$(I_{P_1}+I_{P_2})_d,$$
where $P_1 = [L_1^{d_1} \cdots L_r^{d_r}]$ and $P_2 = [N_1^{d_1} \cdots N_r^{d_r}]$ are points of $\X_M$ and the 
$$\{L_i , 1 \leq i \leq r\}, \{N_i, 1 \leq i \leq r\}$$ are general sets of linear forms in $R$.

By Corollary \ref{tangrvar} we obtain
$$I_{P_1} = (L_1^{d_1-1} \cdots L_r^{d_r-1}) \cap (\cap_{1 \leq i <j\leq r} (L_i,L_j)^{d_i+d_j-1} )
,$$
$$I_{P_2} = (N_1^{d_1-1} \cdots N_r^{d_r-1}) \cap (\cap_{1 \leq i <j\leq r} (N_i,N_j)^{d_i+d_j-1} )
.$$
By the exact sequence
\begin{equation}\label{exact sequence}
0 \longrightarrow (I_{P_1} \cap I_{P_2})_d 
\longrightarrow    (I_{P_1} \oplus I_{P_2})_d
\longrightarrow    (I_{P_1} + I_{P_2})_d
\longrightarrow 0,  
\end{equation}
and the fact that we know  that  $\dim (I_{P_i})_d = r(n-1)+1$, ($i=1,2$), 
(see Remark \ref{remark1}), it is enough to find $\dim 
 (I_{P_1} \cap I_{P_2})_d .$
 
 Recall  that the expected dimension of  $\sigma_2 (\mathbb X_M)$ is
  $$\hbox {exp.dim }  \sigma_2 (\mathbb X_M) = \min \{ 2 r(n-1) + 1, 
{ d+n-1 \choose n-1} -1 \}.$$

Note that, if $\dim (I_{P_1} \cap I_{P_2})_d =0$, then by \eqref{exact sequence} the dimension of 
$\sigma_2 (\mathbb X_M)$ is as expected.

Let $V$  be the subscheme of $\mathbb P^{n-1}$ defined by 
$ I_{P_1} \cap I_{P_2} $
and let $f$ be a form of degree $d$ in  $ I_{P_1} \cap I_{P_2}$. Clearly 
$$f= L_1^{d_1-1} \cdots L_r^{d_r-1}\cdot N_1^{d_1-1} \cdots N_r^{d_r-1} \cdot g$$
where $g$ is a form of degree $d-2(d-r)= 2r-d$.

If $2r-d<0$, of course there are no forms of this degree, hence 
$( I_{P_1} \cap I_{P_2})_d=0$ and we are done.
So assume $2r-d\geq 0$.

The form $g$ 
 vanishes on the residual scheme $W$ of $V$ with respect to the $2r$   multiple hyperplanes $\{ L_i=0\}$ and $\{ N_i=0\}$ ($1 \leq i \leq r$). 
It is easy to see that $W$ is defined by the ideal 
$$(\cap_{1 \leq i <j\leq r} (L_i,L_j) )
\cap (\cap_{1 \leq i <j\leq r} (N_i,N_j) ).
$$
Now,  $g$ cannot be divisible by all the $L_i$ and $N_j$ because otherwise it would have degree at least $2r$. 

Without loss of generality assume that $g$ is not divisible by  $L_1$, and let $H$ be the hyperplane defined by $L_1$.
The form $g$ cuts out on $H$ a hypersurface $\mathcal S$ of $H$ having degree $2r-d$, and containing the $r-1$ hyperplanes of $H$ cut out by $L_2, \ldots,L_r$.
Hence, in order for $g$ to exist, $2r-d $ has to be at least $ r-1$, that is, $d\leq r+1$.   But we are assuming $d\geq r+1$, so we get that $d=r+1$.

It follows that $\mathcal S$ has degree $r-1$, contains the $r-1$ hyperplanes of $H$ cut out by $L_2, \ldots,L_r$,  and  contains the trace on $H$ of the schemes defined by the ideals $(N_i,N_j)$. Since the $N_i$ are generic with respect to $H$ and to the $L_i$'s, the only possibility for $g$ to exist is  that the schemes $Y_{i,j}$ defined by the ideals $(N_i,N_j)$ do not intersect $H$.
Since $H \simeq \mathbb P^{n-2} \subseteq \mathbb P^{n-1}$  and 
$Y_{i,j} \simeq \mathbb P^{n-3} \subseteq \mathbb P^{n-1}$, then 
$H \cap Y_{i,j} = \emptyset$ only for $n\leq3$. 

Therefore, we are left with the following cases:

Case 1: $n=3$, $r\geq 3$, $d=r+1$;

Case 2: $n=3$, $r= 2$, $d=3$.

In Case 1, the form $g$ has degree $2r - d = r - 1$ and it should vanish at the
scheme $W$ which is a union of $2{r \choose 2}$ simple points in $\PP^2$. The dimension of the space of homogeneous polynomials of degree $r - 1$ is ${r+1 \choose 2}$, which is smaller than $2{r \choose 2}$, for $r \geq 3$. 
Then, such a $g$ cannot exist.

In Case 2, $M = Z_1^2Z_2$, the form $g$ has degree $2r-d= 1$ and the scheme $W$ is the union of 2  points of $\mathbb P^2$. Hence  $g $ can exist and describes  the line through the two points. It follows that  $\dim (I_{P_1} \cap I_{P_2})_3 =1$. So from \eqref{eq:ideal} we get
$$\dim  (I_{P_1} + I_{P_2})_3 = 9,$$
that is, $\dim \sigma _2 (\mathbb X_M) =8$.
But
$\hbox {exp.dim} \sigma _2 (\mathbb X_M) =9,$ and so 
 $\mathbb X_M$ has $2$-defect $=1$ and we are done.

\end{proof}
%%%%%%%%%%%%%%%%%%%%%%%%%%%%%%%%%%%%%%%%%%%
\begin {rem}
 The exceptional case noted above was observed in  \cite{CGG0} in connection with the study of the secant varieties of the  tangential varieties to Veronese varieties.
\end{rem}

\begin{ex}

  We claim that the hypersurface in $\mathbb P^9$ containing all those cubic forms of $k[x,y,z]$
which can be written $L_1^2L_2+N_1^2N_2$  (with the $L_i,N_i$ linear forms) is precisely the
hypersurface in $\mathbb P^9$ containing all singular cubics. It is well-known that the (closure) of the set of cubic plane curves with a double point is a hypersurface in $\mathbb P^9$. It will be enough to show that every nodal cubic can be written in the desired form  (since cuspidal cubics can, after a change of variables,  always be written in the form $y^3+x^2z$).

First recall (see  \cite {fulton})  that every nodal cubic can, after a change of variables, be written in the form
$$xyz-x^3-y^3.
$$
With a further change of variables given by
$$ x = -X-Y; \ \ \ \  y=X-Y; \ \ \  z= -Z,$$
we get
$$X^2(6Y+Z) +Y^2( 2Y-Z)
.$$

\end {ex}

%%%%%%%%%%%%%%%%%%%%

\section{Inductive results}

We hope that our results above could provide a starting point for further 
investigations on $\M$-decompositions, when the number $s$ of summands increase.
The main obstruction to extend plainly our arguments resides in the fact that when $r$ is close 
to $n$ and $s>2$, then the shape of polynomials lying in the intersection of a tangent space 
$T_{P}$  with the {\it span} of two of more tangent spaces to $\X_M$ is not easy to control.

In this final section, we show how by Theorem \ref{Binary} 
and Theorem \ref{sec2varn}  we can improve our knowledge on the defectivity
of $\M$-decompositions.
We will need some preliminary algebraic remarks.

\begin{prop} \label{alg1} Let $I,J$ be ideals in the polynomial ring
$R=k[x_1,\ldots,x_n]$, both generated by elements of degree $d-1$, for some $d\geq 3$. 
Assume that $I_d\cap J_d=(0)$. Then after adding one variable $y$
to $R$, we still get $IR[y]_d\cap JR[y]_d=(0)$.
\end{prop}
\begin{proof}
By our assumptions on the generators of $I$,
the elements of $IR[y]_d$ are of the form $Ay+B$, with
$A\in I_{d-1}$ and $B\in I_d$. Similarly the elements of $JR[y]_d$ are of the 
form $Cy+D$, with $C\in J_{d-1}$ and $D\in J_d$. Since $y$ is independent mod $R$,
the equality $Ay+B=Cy+D$ yields $A=C$, $B=D$. Hence $A=B=0$ by assumption.
\end{proof}

Next, fix $d\geq 3$ and  go back to the ideals $I_P$ 
defined in the previous sections, 
where $P$ is a splitting form of fixed type:
\[M =Z_1^{d_1}\cdots Z_r^{d_r},\qquad d=d_1+\dots +d_r.\]

\begin{prop} \label{alg2} Assume that for a general choice of
splitting forms $P_1,\dots, P_s\in k[x_1,\ldots ,x_t]$ of type $M$ in $t < n$ variables
the ideal $I=I_{P_1}+\dots +I_{P_s}\subset k[x_1,\ldots ,x_t]$ has the degree $d$ piece
of (maximal) dimension $s(r+1)$.
Then for a general choice of splitting forms $Q_1,\dots, Q_s\in R$ 
of type $M$, the ideal $J=I_{Q_1}+\dots +I_{Q_s}\subset R$ has a degree $d$ piece
of (maximal) dimension $sr(n-1)+s$. 
\end{prop}

\begin{proof}
We make induction on $s$. Notice that it is enough to prove that 
 the degree $d$ piece of
the extended ideal $IR$ has dimension $sr(n-1)+s$, because then,
by semicontinuity, the same will hold for a general choice of the $Q_i$'s.

For the case $s=1$, take $P_1=L_1^{d_1}\cdots L_r^{d_r}$,
where the $L_i$'s are general linear forms in $k[x_1,\ldots ,x_t]$. 
The vector space $(IR)_d$ is spanned by  the $(t-1)r+1$ independent forms in the $t$ variables $x_1, \ldots ,x_t$ which are a basis for $I_d$, plus the $r (n-t)$ independent forms $x_iP/L_j$ for all $i=t+1,\dots n$ and $j=1,\dots r$. It follows that $\dim (IR)_d =(t-1)r+1+r (n-t)= r(n-1)+1.$

For higher $s$, notice that our assumption implies that $I'=I_{P_1}+\dots +I_{P_{s-1}}$ has a null intersection with $I_{P_s}$
in degree $d$, inside $k[x_1,\ldots ,x_t]$. Thus  also the extensions of
$I'$ and $I_{P_s}$ to $R$ have a null intersection (apply $n-t$ times the previous
Proposition \ref{alg1}). Since, by the inductive hypothesis, 
$\dim( I' R)_d =(s-1) r (n-1) + (s-1)$, we have that
$$\dim( I R)_d =(s-1) r (n-1) + (s-1) +r (n-1)+1= s r (n-1) + s,$$
and
the result follows.

\end{proof}

Now we can apply the previous proposition to extend 
 Theorem \ref{Binary} to many variables.
 
\begin{prop} \label{ext1} Let $R = k[x_1,\dots,x_n]$ and 
let $M = Z_1^{d_1}\cdots Z_r^{d_r} \in {\mathcal M}_{r,d}$, $r \leq d$.
Assume $s(r+1)\leq d+1$.
 Then, for every $M$, 
$\sigma_s({\mathbb X}_M)$  has the expected  dimension 
\[
\dim \sigma_s({\mathbb X}_M) = s\dim {\mathbb X}_M + (s-1)=sr(n-1)+s-1.
\]

\end{prop}
\begin{proof}
The result is true in two variables, by Theorem \ref{Binary}. 
In more than two variables, specialize the $s$ points $P_i$'s to forms
in two variables: the result remains true by Proposition \ref{alg2}.
\end{proof}

The condition $s(r+1)\leq d+1$ is fundamental in the previous argument,
because we need that, after the specialization, each $I_{P_i}$
has a null intersection with the span of the remaining $I_{P_j}$  (in $k[x_1,x_2]$).
We can improve the previous result with the trick of separating the variables. 

\begin{prop} \label{alg3} Split the variables of $R=k[x_1,\dots,x_n]$
in two sets by defining $R_1=k[x_1,\dots,x_j]$ and $R_2=k[x_{j+1},\dots, x_n]$.
Let $I\subset R_1$, $J\subset R_2$ be ideals generated by elements of degree
$d-1$, for some $d\geq 3$. Then $(IR)_d\cap( JR)_d=(0)$.
\end{prop}
\begin{proof}
Notice that the elements of $(IR)_d$ have degree at least $d-1$ in
$x_1,\dots, x_j$, while the elements of $(JR)_d$ have degree at most $1$ in
$x_1,\dots, x_j$. 
Since $d \geq 3$,  the two things cannot match, and the claim follows.
\end{proof}

Now, by dividing the points $P_1,\dots,P_s$ in 
groups and specializing each group to forms in two distinct variables, we get

\begin{prop} \label{ext2} Let $R = k[x_1,\dots,x_n]$ and 
let $M = Z_1^{d_1}\cdots Z_r^{d_r} \in {\mathcal M}_{r,d}$, $r \leq d$,
and $d\geq3$.
Set $m = \left\lfloor \frac{n}{2} \right\rfloor$ and 
$s'=\left\lfloor\frac{d+1}{r+1}\right\rfloor$.
Then, for every $s\leq s'm$ and for every  $M$, $\sigma_s({\mathbb X}_M)$  
has the expected  dimension 
\[
\dim \sigma_s({\mathbb X}_M) = s\dim {\mathbb X}_M + (s-1)=sr(n-1)+s-1.
\]

\end{prop}
\begin{proof}
It is enough to prove the statement for $s=s'm$.
For each $j=1,\dots,m$ take $s'$ forms $P_{1j},\dots, P_{s'j}\in \X_M$ 
in the variables $x_{2j-1}, x_{2j}$. The ideal 
\[I_j=(I_{P_{1j}}+\dots +I_{P_{s'j}})\]
in $k[x_{2j-1}, x_{2j}]$ has the expected dimension $s'r+s'$ 
in degree $d$.  By Proposition \ref{alg2} with $t=2$,
the extension of $I_j$ to $R$ 
has the expected dimension $s'r(n-1)+s'$ in degree $d$.

Then the sum $I=I_1+\dots + I_m$ has the expected dimension
$m(s'r(n-1)+s')=sr(n-1)+s$ in degree $d$, because by Proposition \ref{alg3} every
$I_j$ has null intersection with the sum of the remaining $I_k$'s, $k\neq j$.
\end{proof}

We can use the idea of separating the variables also
in connection with  Theorem \ref{sec2varn}, by specializing pairs of points $P_i$'s
to forms in three distinct variables. The statement below covers some cases which 
do not fit with the numerical assumptions of Proposition \ref{ext2}.

\begin{prop}\label{ext3} Let $R= k[x_1,\ldots,x_n]$,  $r\geq 2$, $d \geq 3$.
Let $M \in \mathcal M_{r,d}$ be different from  $Z_1^2 Z_2$.
Then $\sigma_s (\mathbb X_M)$  is not defective for 
$ s \leq 2 \left\lfloor \frac{n}{3} \right\rfloor$ .

\end{prop}
\begin{proof}
The case $s=2$ is provided by Theorem \ref{sec2varn}. For higher $s$,
it sufficies to prove the statement for $ s = 2 \left\lfloor \frac{n}{3} \right\rfloor$.
For all $i=1,...,\left\lfloor \frac{n}{3} \right\rfloor$ choose general points $P_{2i-1},P_{2i}\in \X_M$ which involve only the 
variables $x_{3i-2},x_{3i-1},x_{3i}$. 
By Theorem \ref{sec2varn} and its proof, we know
that the ideals $I_{P_{2i+1}}, I_{P_{2i+2}}$ have null intersection in degree $d$, considered
as ideals in $k[x_{3i+1},x_{3i+2},x_{3i+3}]$. Thus their sum $I_i = I_{P_{2i-1}} + I_{P_{2i}}$
  has the expected dimension
in degree $d$. The same remains true when we extend the ideals to $R$, 
by Proposition \ref{alg2} with $t=3$. Now it is enough to apply Proposition \ref{alg3}  to show that
the sum of the $I_i$'s has the expected dimension in degree $d$.
 \end{proof}

The results proved in Proposition \ref{ext2} or 
Proposition \ref{ext3} by separating
the variables are clearly not optimal, in the sense that they cannot cover
the full range of all possible $M$-ranks. In order to get a complete classification of
defective varieties $\X_M$ for $s>2$, new ideas should be introduced.

\section*{Acknowledgements} 

The first author wishes to thank Queen’s University, in the person of the third author, 
for their kind hospitality during the preparation of this work. The first and  third authors enjoyed 
support from NSERC (Canada).  The first and second authors  were also supported by GNSAGA of INDAM and by MIUR funds 
(PRIN 2010-11 prot. 2010S47ARA-004 - Geometria delle Variet\`a Algebriche) (Italy). 
The fourth author was supported by Jubileumsdonationen, K \& A Wallenbergs Stiftelse (Sweden) during visits to the first and the third author.

% ------------------------------------------------------------------------

\begin{thebibliography}{1}
\bibitem{arrondobernardi} E. Arrondo and A. Bernardi, {\it On the variety parameterizing completely decomposable polynomials}, J. Pure Appl. Algebra, \textbf{215} (2011), 201--220.

\bibitem{abo} H. Abo, {\it Varieties of completely decomposable forms and their secants}, J. of Algebra, \textbf{403} (2014), 135--153.

\bibitem{aboV} H. Abo and N. Vannieuwenhoven, {\it Most secant varieties of tangential varieties to Veronese varieties are nondefective}, arXiv preprint arXiv:1510.02029, (2015).

\bibitem{AH95} J. Alexander and A. Hirschowitz, {\it Polynomial interpolation in several variables}, J. of Algebraic Geom., \textbf{4} (1995), 201--222.

\bibitem{ballico2005} E. Ballico, {\it On the secant varieties to the tangent developable of a Veronese variety}, J. Algebra, \textbf{288} (2005), 279--286.

\bibitem{BCGI2009} A. Bernardi, M.V. Catalisano, A. Gimigliano and M. Id\'a, {\it Secant varieties to osculating varieties of Veronese embeddings of $\mathbb{P}^n$}, J. Algebra, \textbf{321} (2009), 982--1004.

\bibitem{briand} E. Briand, {\it Covariants vanishing on totally decomposable forms}, Liaison, Schottky problem and invariant theory, Birkhäuser Verlag, Basel, \textbf{280} (2010), 237--256.

\bibitem{Cayley} A. Cayley, {\it The collected mathematical papers. Volume 2}, Cambridge University Press, Cambridge, (2009), ii+xii+606.

\bibitem{CaCh} E. Carlini and J. Chipalkatti, {\it On Waring's problem for several algebraic forms}, Comment. Math. Helv., \textbf{78} (2003), 494--517.

\bibitem{carcatgermonomi} E. Carlini, M.V. Catalisano and A.V. Geramita, {\it The solution to the Waring problem for monomials and the sum of coprime monomials}, J. of Algebra, \textbf{370} (2012), 5--14.

\bibitem{ChGe} L. Chiantini and A.V. Geramita, {\it Expressing a general form as a sum of determinants}, Collect. Math., \textbf{66} (2015), 227--242.

\bibitem{magnifici} M.V. Catalisano, A.V. Geramita, A. Gimigliano, B. Harbourne, J. Migliore, U. Nagel and Y. Shin, {\it Secant Varieties of the Varieties of Reducible Hypersurfaces in $\mathbb{P}^n$}, arXiv preprint arXiv:1502.00167 (2015).

\bibitem{CGG0} M.V. Catalisano, A.V. Geramita and A. Gimigliano, {\it On the secant varieties to the tangential varieties of a Veronesean}, Proc. Amer. Math. Soc., \textbf{130} (2002), 975--985.

\bibitem{Chi02} J. Chipalkatti, {\it Decomposable ternary cubics}, Experiment. Math., \textbf{11} (2002), 69--80.

\bibitem{Chi03} J. Chipalkatti, {\it On equations defining coincident root loci}, J. of Algebra, \textbf{267} (2003), 246--271.

\bibitem{Chi04} J. Chipalkatti, {\it Invariant equations defining coincident root loci}, Arch. Math. (Basel), \textbf{83} (2003), 422--428.

\bibitem{Chi04bis} J. Chipalkatti, {\it The Waring locus of binary forms}, Comm. Algebra, \textbf{32} (2004), 1425--1444.

\bibitem{ComMour} P. Comon and B. Mourrain, {\it Decomposition of Quantics in sums of power of linear forms}, Signal Processing, \textbf{53} (1996), 93--107.

\bibitem{comon} P. Comon, {\it Tensor decompositions: state of the art and applications}, Mathematics in signal processing, V (Coventry, 2000), Oxford Univ. Press, Oxford, \textbf{71} (2002), 1--24.

\bibitem{fulton} W. Fulton, {\it Algebraic curves}, Addison-Wesley Publishing Company, Advanced Book Program, Redwood City, CA, (1989), xxii+226.

\bibitem{GeSc} A.V. Geramita and H. Schenck, {\it Fat points, inverse systems, and piecewise polynomial functions}, J. of Algebra, \textbf{204} (1998), 116--128.

\bibitem{henrion} D. Henrion and J.-B. Lasserre, {\it Inner approximations for polynomial matrix inequalities and robust stability regions}, IEEE Trans. Automat. Control, \textbf{57} (2012), 1456--1467.

\bibitem{IaKa} A. Iarrobino and V. Kanev, {\it Power sums, Gorenstein algebras, and determinantal loci}, Springer-Verlag, (1999).

\bibitem{LandsbergTeitler2010} J.M. Landsberg and Z. Teitler, {\it On the ranks and border ranks of symmetric tensors}, Found. Comput. Math., \textbf{10} (2010), 339--366.

\bibitem{Pucci} M. Pucci, {\it The Veronese Variety and Catalecticant Matrices}, J. of Algebra, \textbf{202} (1998), 72--95.

\bibitem{Sal3} G. Salmon, {\it Higher Algebra}, Chelsea Publishing Co., New York, (1964).

\bibitem{shin2012} Y. Shin, {\it Secants to the variety of completely reducible forms and the Hilbert function of the union of star-configurations}, J. of Algebra and its Applications, \textbf{11} (2012), 1250109.

\bibitem{Syl} J.J. Sylvester, {\it Collected Mathematical Papers}, Cambridge University Press, (1904), I-IV.

\bibitem{Terracini} A. Terracini, {\it Collected Mathematical Papers}, Cambridge University Press, (1904), I-IV.

\bibitem{torrance} D.A. Torrance, {\it Generic forms of low Chow rank.} arXiv preprint arXiv:1508.05546 (2015).

\end{thebibliography}
\end{document}